\documentclass[reqno,11pt]{amsart}
\numberwithin{equation}{section}
\usepackage{amsmath}
\usepackage{esint} 
\usepackage{amsthm}
\usepackage{epsfig}
\usepackage{psfrag}
\usepackage{graphicx}
\graphicspath{{figure/}}
\usepackage{graphpap,latexsym,epsf}
\usepackage{color}
\usepackage{amssymb,mathrsfs,enumerate}
\usepackage{mathtools}
\definecolor{citegreen}{rgb}{0,0.6,0}
\definecolor{refred}{rgb}{0.8,0,0}
\usepackage[colorlinks, citecolor=citegreen, linkcolor=refred]{hyperref}

\def\RRR{{\mathrm R}}

\def\a{\alpha}
\def\b{\beta}

\newcommand{\pa}{\partial}

\newcommand{\Ric}{{\rm Ric}}

\newcommand{\De}{\Delta}

\newcommand{\na}{\nabla}
\newcommand{\nana}{\nabla^2}

\mathchardef\emptyset="001F

\definecolor{vgreen}{rgb}{0.1,0.5,0.2}
\definecolor{viola}{RGB}{85,26,139}

\newtheorem{theorem}{Theorem}[section]
\newtheorem{remark}{Remark}

\newtheorem{proposition}[theorem]{Proposition}

\newtheorem{lemma}[theorem]{Lemma}

\author[S.~Borghini]{Stefano Borghini}
\address{S.~Borghini, Universit\`a degli Studi di Trento,
via Sommarive 14, 38123 Povo (TN), Italy}
\email{stefano.borghini@unitn.it}

\author[L.~Mazzieri]{Lorenzo Mazzieri}
\address{L.~Mazzieri, Universit\`a degli Studi di Trento,
via Sommarive 14, 38123 Povo (TN), Italy}
\email{lorenzo.mazzieri@unitn.it}

\begin{document}

\hyphenation{ma-ni-fold}

\title[Counterexamples to a divergence lower bound for the covariant derivative of skew-symmetric 2-tensor fields
]{Counterexamples to a divergence lower bound for the covariant derivative of skew-symmetric 2-tensor fields
}

\begin{abstract} 
In~\cite{Hwang_Yun} an estimate for suitable skew-symmetric $2$-tensors was claimed. Soon after, this estimate has been exploited to claim powerful classification results: most notably, it has been employed to propose a proof of a Black Hole Uniqueness Theorem for vacuum static spacetimes with positive scalar curvature~\cite{Xu_Ye} and in connection with the Besse Conjecture~\cite{Hwang_Yun_1}. In the present note we point out an issue in the argument proposed in~\cite{Hwang_Yun} and we provide a counterexample to the estimate. 
\end{abstract}

\maketitle


%

\maketitle

\section{Introduction}

The Black Hole Uniqueness Theorem for three-dimensional static solutions with positive scalar curvature and the Besse Conjecture for solutions to the Critical Point Equation are two very famous and related open problems in contemporary geometric analysis. Very recently, some very remarkable advances have been claimed on both of these problems in a series of papers~\cite{Hwang_Santos_Yun,
Hwang_Yun_2,Hwang_Yun, Xu_Ye,Gabjin_Seungsu,Hwang_Yun_1}.
In this short note, we point out an issue in the approach proposed in the above mentioned papers, providing counterexamples.

To introduce the problems of interest together with some notation, let us recall that a three-dimensional static solution is a triple $(M,g,f)$ satisfying
\begin{equation}
\label{eq:Static}
f \Ric \,=\,\nana f+ \frac{\RRR}{2}\,  f \,g\,,\qquad\De f\,=\,-\frac{\RRR}{2}f \, ,
\end{equation}
where $(M,g)$ is a Riemannian manifold, $f$ is a smooth function and $\Ric$ and $\RRR$ denote the Ricci tensor and the scalar curvature of $g$, respectively. When $\RRR$ is positive, it is natural to suppose that $(M,g)$ is a compact manifold with boundary and that $f$ is vanishing on the boundary.
A strictly related problem is the so called Critical Point Equation, which consists in the following system 
\begin{equation}
\label{eq:Besse}
(1+f) \left( \Ric - \frac{\RRR}{n} \, g \right)\,=\,\nana f+\frac{\RRR}{n(n-1)}\,g\,,\qquad\De f\,=\,-\frac{\RRR}{n-1}f
\end{equation}
where the unknowns are given by the triple $(M, g, f)$, with $(M,g)$ a closed Riemannian manifold and $f$ a smooth function.

In~\cite{Hwang_Yun}, the authors aim at classifying solutions to the Critical Point Equation subject to the condition of having Positive Isotropic Curvature. To this end, they consider the differential $2$-form 
$$
\omega\,=\,df\wedge \iota_{\na f}z\,,
$$
where $z$ indicates the traceless Ricci tensor, and they claim that it must vanish. Notice that, using~\eqref{eq:Besse}, the differential $2$-form $\omega$ can be rewritten as
\begin{equation*}
\omega\,=\,\frac{1}{2(1+f)}df\wedge d|\na f|^2\,,
\end{equation*}
where $|\cdot|$ is the norm computed with respect to the metric $g$. If $\omega \equiv 0$, then, using again the equation~\eqref{eq:Besse}, one can prove that the Cotton tensor of $g$ must also vanish, by a direct computation. 
It follows that either $n=3$ and $g$ is Locally Conformally Flat, or else $n \geq 4$ and $g$ has harmonic Weyl tensor. In both cases, the classification follows easily. The same strategy is adopted in~\cite{Xu_Ye}\footnote{Notice that this reference has been withdrawn by the authors during the preparation of the present note.}, where this time the differential $2$-form $\omega$ is defined as
\begin{equation*}
\omega\,=\,\frac{1}{2f}df\wedge d|\na f|^2\,,
\end{equation*}
with $g$ and $f$ satisfying~\eqref{eq:Static}. In both cases, the vanishing of $\omega$ is deduced through an integration by parts argument -- which we describe in Subsection~\ref{sub:byparts} below, in the case of static metrics -- making a substantial use of the key estimate
\begin{equation}
\label{eq:main}
|\na\omega|^2\,\geq\,|\delta\omega|^2\, \, ,
\end{equation}
which the authors claim to hold at all points of $M$ where $\omega$ is not vanishing (see Lemma 5.5 in~\cite{Hwang_Yun}).
The proposed proof of~\eqref{eq:main} does not make use of the full strength of either~\eqref{eq:Static} or~\eqref{eq:Besse}. In fact, it is based on a local computation, in which the global structure of $M$ is not playing any role. As such, if correct, it should work for every differential $2$-form having the structure
\begin{equation}
\label{eq:Pi}
\omega \,=\,\lambda(f)\,df\wedge d|\na f|^2\,.
\end{equation}
for some smooth function $\lambda=\lambda(f)$, independently of the validity of~\eqref{eq:Static} or~\eqref{eq:Besse}. Aim of the present note is to disprove the claim that every $\omega$ as in~\eqref{eq:Pi}, defined on an open subset of a Riemannian manifold $(M,g)$, satisfies estimate~\eqref{eq:main}. 

In Section~\ref{sec:issue} we point out the issue in the original proof of~\eqref{eq:main}, given in~\cite[Lemma~5.5]{Hwang_Yun}, whereas in Section~\ref{sec:counterexample} we provide effective counterexamples to the claim. Namely, we show that

\medskip

\noindent {\em 
For every smooth real function $\lambda \not\equiv 0$, there exist a smooth Riemannian metric $g$ and a smooth function $f$ such that $|\nabla \omega|^2 < |\delta \omega|^2$, with $\omega =\,\lambda(f)\,df\wedge d|\na f|^2$.
}

\medskip

For the sake of completeness, we discuss in Section~\ref{sec:properties} how the validity of an estimate like~\eqref{eq:main} can be exploited to deduce that $\omega$ must vanish everywhere.

\section{Analysis of a skew-symmetric $2$-tensor field}
\label{sec:properties}

To make our computations more transparent, we prefer to work with the tensor-fields formalism. However one can also work with the formalism of differential forms as done in~\cite{Hwang_Yun}. Instead of $\omega$ defined as in~\eqref{eq:P}, we consider the skew-symmetric $2$-tensor field $P$, given by
\begin{equation}
\label{eq:P}
P \,=\,\lambda(f)\, \left[ df \otimes d|\na f|^2 - d|\na f|^2  \otimes df \right] \,,
\end{equation}
with $\lambda$, $f$ and $g$ as above. In this formalism, we have that estimate~\eqref{eq:main} is equivalent to
\begin{equation}
\label{eq:mainP}
|\na P|^2\,\geq\, 2 \, | {\rm div} P|^2\, \, ,
\end{equation}
as $2 \, |\na \omega|^2 = |\na P|^2$ (the factor two comes from the slight difference in the definition of norms on differential forms and tensor, namely $|\na\omega|^2=\sum_{j<k}\sum_i (\na_i\omega_{jk})^2$, whereas $|\na P|^2=\sum_{j,k}\sum_i (\na_i P_{jk})^2$) and $ \delta \omega = - {\rm div} P$. Notice that, replacing the constant $2$ with the smaller constant $1/n$, one gets the always valid lower bound $|\na P|^2\,\geq\, (1/n) \, | {\rm div} P|^2$. Furthermore, exploiting the special structure~\eqref{eq:P} of $P$, one can significantly improve on this bound, obtaining $(n-1)|\na P|^2\,\geq\, 2 \, | {\rm div} P|^2$ (see the appendix). On the other hand, estimate~\eqref{eq:mainP} is too strong and cannot hold in general, as we will discuss below.

\subsection{Two differential identities.}
Here we discuss some basic though fundamental properties of a skew-symmetric $2$-tensor $P$ having the form~\eqref{eq:P}. 

\begin{proposition}
\label{pro:closed}
Let $(M,g)$ be a $n$-dimensional Riemannian manifold and let $f\in\mathscr{C}^\infty(M)$. Then, the skew-symmetric $2$-tensor field $P$ defined as in~\eqref{eq:P}, for some smooth real function $\lambda$, satisfies the identity
\[
\na P(X,Y,Z)\,+\,\na P(Y,Z,X)\,+\,\na P(Z,X,Y)\,=\,0\,.
\]
\end{proposition}

\begin{remark}
In terms of the differential $2$-form $\omega$ defined as in~\eqref{eq:Pi}, the above identity is telling us that $\omega$ is closed, as observed in~\cite[Lemma~5.4]{Hwang_Yun}. Observe that, if $\omega$ is as in~\eqref{eq:Pi}, then it is straightforward to realize that $d\omega = (d\lambda/df) \, df \wedge df \wedge d |\na f|^2 = 0$.
\end{remark}

\begin{proof}
For simplicity we work with normal coordinates $\{x^1,\dots,x^n\}$. A simple computation gives
\begin{align*}
\na_i P_{jk}\,=\,\frac{\dot\lambda}{\lambda}P_{jk}\na_i f+\lambda\left(\nana_{ij}f\na_k|\na f|^2-\na_j |\na f|^2 \nana_{ik} f\right)+\lambda\left(\na_j f\nana_{ik}|\na f|^2-\nana_{ij}|\na f|^2\na_k f\right)\,.
\end{align*}
It is now a matter of computation to check that the sums over rotating indexes of the three pieces on the right hand side give zero. We compute
\begin{multline*}
P_{jk}\na_i f+P_{ki}\na_j f+P_{ij}\na_k f\,=\,
\lambda\Big(\na_i f\na_j f\na_k|\na f|^2-\na_i f\na_j|\na f|^2\na_k f
\\
+\na_j f\na_k f\na_i|\na f|^2-\na_j f\na_k|\na f|^2\na_i f+\na_k f\na_i f\na_j|\na f|^2-\na_k f\na_i|\na f|^2\na_j f
\Big)\,=\,0\,.
\end{multline*}
Similarly, one has
\[
\nana_{ij}f\na_k|\na f|^2-\na_j |\na f|^2 \nana_{ik} f+
\nana_{jk}f\na_i|\na f|^2-\na_k |\na f|^2 \nana_{ji} f+
\nana_{ki}f\na_j|\na f|^2-\na_i |\na f|^2 \nana_{kj} f
=0\,, \]
\[
\na_j f\nana_{ik}|\na f|^2-\nana_{ij}|\na f|^2\na_k f+
\na_k f\nana_{ji}|\na f|^2-\nana_{jk}|\na f|^2\na_i f+
\na_i f\nana_{kj}|\na f|^2-\nana_{ki}|\na f|^2\na_j f
=0\,.
\] 

It follows then that
\[
\na_i P_{jk}+
\na_j P_{ki}+
\na_k P_{ij}\,=\,0\,,
\]
as claimed.
\end{proof}

Another interesting property of $P$ is that it satisfies a Bochner-type formula, as it is established in the following proposition.

\begin{proposition}
Let $(M,g)$ be a $n$-dimensional Riemannian manifold and let $f\in\mathscr{C}^\infty(M)$. Then, the skew-symmetric $2$-tensor field $P$ defined as in~\eqref{eq:P}, for some smooth real function $\lambda$, satisfies the identity
$$
\frac{1}{2}\De|P|^2\,=\,|\na P|^2+2\langle P\,|\,\na ({\rm div}\,P)\rangle+\frac{2\RRR}{(n-1)(n-2)}|P|^2+2\frac{n-4}{n-2}\RRR_{js}P_{sk} P_{jk}+2{\rm W}_{ijks}P_{is}P_{jk}\,.
$$
\end{proposition}

\begin{proof}
We perform our computations with respect to normal coordinates.
Exploiting Proposition~\ref{pro:closed} and the skew-symmetry of $P$, we compute
\begin{align*}
\De |P|^2\,&=\,2\na_i(P_{jk}\na_i P_{jk})
\\
&=\,2|\na P|^2+2P_{jk}\De P_{jk}
\\
&=\,2|\na P|^2-2P_{jk}\nana_{ij} P_{ki}-2P_{jk}\nana_{ik}P_{ij}
\\
&=\,2|\na P|^2+4P_{jk}\nana_{ij} P_{ik}
\\
&=\,2|\na P|^2+4P_{jk}\left(\nana_{ji} P_{ik}+\RRR_{ijis}P_{sk}+\RRR_{ijks}P_{is}\right)
\\
&=\,2|\na P|^2+4P_{jk}\left(\na_j({\rm div}\, P)_k+\RRR_{js}P_{sk}+\RRR_{ijks}P_{is}\right)
\,.
\end{align*}
To obtain the claimed identity, it is now enough to substitute the general formula for the Riemann tensor
$$
\RRR_{ijks}=-\frac{\RRR}{(n-1)(n-2)}(g_{ik} g_{js}-g_{is}g_{jk})+\frac{1}{n-2}\left(\RRR_{ik} g_{js}-\RRR_{is}g_{jk}+g_{ik} \RRR_{js}-g_{is}\RRR_{jk}\right)+\mathrm{W}_{ijks}
$$
in the computation above.
\end{proof}

The differential identity obtained in the previous proposition simplifies significantly when $n=3$, since in this case the Weyl tensor vanishes and we get
\begin{equation}
\label{eq:bochner}
\frac{1}{2}\De|P|^2\,=\,|\na P|^2+2\langle P\,|\,\na ({\rm div}\,P)\rangle+\RRR|P|^2-2\RRR_{js}P_{sk} P_{jk}\,.
\end{equation}

\subsection{Application to $3$-dimensional static solutions}
\label{sub:byparts}

In~\cite{Xu_Ye} a classification result for $3$-dimensional static metrics with positive scalar curvature was proposed, building on the above Bochner-type formula and on the validity of estimate~\eqref{eq:false}. For completeness, here we retrace their proof.

Using formula~\eqref{eq:Static}, we can substitute the Ricci tensor in~\eqref{eq:bochner}, getting
\begin{equation}
\label{eq:einsteintype_bochner}
\frac{1}{2}\De|P|^2\,=\,|\na P|^2+2\langle P\,|\,\na ({\rm div}\,P)\rangle+\frac{\RRR}{2}|P|^2+\frac{2}{f}P(\na f,{\rm div}\,P)-\frac{1}{2f}\langle\na f\,|\,\na|P|^2\rangle\,,
\end{equation}
which can be rewritten as
$$
\frac{1}{2}{\rm div}(f|P|^2)\,=\,f|\na P|^2+2f\langle P\,|\,\na ({\rm div}\,P)\rangle+\frac{\RRR}{2}\,f\,|P|^2+2P(\na f,{\rm div}\,P)\,.
$$
Since $M$ is compact and $f=0$ on $\pa M$, integrating by parts we obtain then
$$
0\,=\,\int_M\left[f|\na P|^2-2f|{\rm div}\,P|^2+\frac{\RRR}{2}\,f\,|P|^2\right]d\mu\,.
$$
Here one can appreciate the strength of estimate~\eqref{eq:mainP}. Indeed, if~\eqref{eq:mainP} is in force and $\RRR>0$, then $|P|^2$ must vanish identically and we obtain the following
\begin{proposition}
Let $(M,g,f)$ be a compact three-dimensional static solution with positive scalar curvature and nonempty boundary. Assume that $f=0$ on $\pa M$ and positive in the interior. If estimate~\eqref{eq:mainP} holds for some $P$ as in~\eqref{eq:P}, then $P$ must vanish identically and one has
$$
df \otimes d|\na f|^2 = d|\na f|^2  \otimes df \, .
$$
\end{proposition}
This is a crucial step in the strategy outlined in~\cite{Xu_Ye}. As anticipated,  they exploit the identity $P=0$ in combination with the static equation to deduce that the Cotton tensor must vanish. The classification follows, 
invoking a well known result by Kobayashi~\cite{Kobayashi} and Lafontaine~\cite{Lafontaine}.

As we are going to see in the next sections, it is not clear how to establish the validity of~\eqref{eq:mainP} in general, however we will prove in the appendix that the weaker lower bound $|\na P|^2\,\geq\, | {\rm div} P|^2$ holds true. This leads to 
$$
\int_M  f|{\rm div}\,P|^2 d\mu \, \geq
 \int_M \frac{\RRR}{2}\,f\,|P|^2  d\mu\,.
$$ 
Building on this integral inequality, one might classify three-dimensional static metrics with positive scalar curvature admitting a divergence-free $P$-tensor.

\section{The issue in the proof of the estimate}
\label{sec:issue}

Here we retrace the proof of estimate~\eqref{eq:main} originally proposed in~\cite[Lemma~5.5]{Hwang_Yun}, pointing out the main issue in the argument. 

As a first step, the authors find a local orthonormal frame with respect to which the tensor $P$ has a nice structure. This part of the proof appears to be correct and it is an interesting fact on its own that will also be helpful in the appendix, so we include it here as a lemma. In the following statement it is helpful to consider the vector valued 1-form $A:TM\to TM$ defined by $P(X,Y)=g(AX,Y)$. In coordinates: $A_i^j=g^{jm}P_{im}$. 
 
\begin{lemma}
\label{le:frame}
Let $(M,g)$ be a $n$-dimensional Riemannian manifold. Let $f\in\mathscr{C}^\infty(M)$ and let $P$ be the tensor defined by~\eqref{eq:P}. Let $x\in M$ be a point with $|P|(x)\neq 0$. Then in a small neighborhood $U$ of $x$ it holds $|P|\neq 0$, $|\na f|\neq 0$, $|A\na f|\neq 0$ and there exists a smooth orthonormal frame $\{E_1,\dots,E_n\}$ with $E_1=\na f/|\na f|$ and $E_2=AE_1/|AE_1|$. With respect to this frame, the tensor $P$ rewrites as
\begin{equation}
\label{eq:frame}
P\,=\,u\left(\theta^1\otimes\theta^2-\theta^2\otimes\theta^1\right)\,,
\end{equation}
where $u$ is a smooth function and $\{\theta^1,\dots,\theta^n\}$ is the dual coframe of $\{E_1,\dots,E_n\}$ (namely, $\theta^i(E_j)=\delta_j^i$ at any point in $U$).
\end{lemma}

\begin{proof}
A proof of this fact is given in~\cite{Hwang_Yun}, however we write here a shorter self contained version.

We first construct the orthonormal frame in the lemma. Consider a neighborhood $U$ of $x$ in which $|P|\neq 0$. From the definition~\eqref{eq:P} of $P$, it is clear that $|\na f|\neq 0$ in $U$ as well. In particular the vector $E_1=\na f/|\na f|$ is well defined in $U$. We complete $E_1$ to an orthonormal frame $\{E_1,\widetilde{E}_2,\dots,\widetilde{E}_n\}$ in $U$. Since $g(E_1,\widetilde{E}_i)=0$ for $i\geq 2$, we have $\na_{\widetilde{E}_i} f=0$ for any $i\geq 2$, hence
\begin{equation}
\label{eq:P_ij}
P(\widetilde{E}_i,\widetilde{E}_j)\,=\,
\lambda(f)\left(\na_{\widetilde{E}_i} f\,\na_{\widetilde{E}_j}|\na f|^2\,-\,\na_{\widetilde{E}_i}|\na f|^2\,\na_{\widetilde{E}_j} f\right)\,=\,0\,,
\end{equation}
for any $i,j\geq 2$. Since $|P|\neq 0$ in $U$, then at any point in $U$ it holds $g(AE_1,\widetilde{E}_j)=P(E_1,\widetilde{E}_j)\neq 0$ for some $j$. In particular $AE_1\neq 0$ in $U$. Since $g(AE_1,E_1)=P(E_1,E_1)=0$, it follows that $AE_1$ is orthogonal to $E_1$. In particular, the vector $E_2=AE_1/|AE_1|$ is well defined and orthonormal to $E_1$ on the whole $U$. We can then complete $E_1,E_2$ to an orthonormal frame $\{E_1,\dots,E_n\}$ in $U$. This is precisely the orthonormal frame described in the statement of the lemma.
Notice in particular that
$$
P(E_1,E_j)\,=\,g(AE_1,E_j)\,=\,|AE_1|\,g(E_2,E_j)\,=\,|AE_1|\,\delta_{2j}\,.
$$
In view of~\eqref{eq:P_ij}, we deduce that the only nonzero entries of $P$ are $P(E_1,E_2)=-P(E_2,E_1)$. Formula~\eqref{eq:frame} follows.
%
%
%
%
\end{proof}

Next, the authors compute $|\na P|^2$ and $|{\rm div}\,P|^2$ with respect to this frame.  The computations regarding  $|\na P|^2$ appear to be correct. On the other hand, it seems to us that the expression of the divergence term worked out by the authors contains a mistake. A simple calculation (see the appendix for more details) gives
\begin{equation}
\label{eq:true}
\begin{aligned}
({\rm div}\,P)(E_1)\,&=\,-E_2(u)\,+\,\sum_{i=3}^n\langle \na_{E_i}E_i\,|\,E_2\rangle \,u
\,=\,-E_2(u)\,+\,\sum_{i=3}^n\langle E_i\,|\,[E_2,E_i]\rangle \,u\,,
\\
({\rm div}\,P)(E_2)\,&=\,E_1(u)\,-\,\sum_{i=3}^n\langle \na_{E_i}E_i\,|\,E_1\rangle \,u\,
\,=\,E_1(u)\,+\,\sum_{i=3}^n\langle E_i\,|\,[E_1,E_i]\rangle \,u\,,
\\
({\rm div}\,P)(E_k)
\,&=\,\langle E_k\,|\,[E_1,E_2]\rangle \,u\,,\quad k\geq 3\,.
\end{aligned}
\end{equation}
It is worth pointing out that the frame $\{E_1,\dots,E_n\}$ was constructed with a pointwise argument. The frame is easily seen to be smooth, but it is important to observe that it is not necessarily induced from a local coordinate system. In particular, the Lie brackets $[E_i,E_j]$ are not necessarily vanishing. This seems to be the core of the issue: in fact, the authors claim that 
\begin{equation}
\label{eq:false}
{\rm div}\,P\,=-E_2(u)\theta^1+E_1(u)\theta^2\,.
\end{equation}
In view of~\eqref{eq:true}, this formula appears to be incorrect whenever the Lie brackets do not vanish.

\begin{remark}
In~\cite{Hwang_Yun}, and more precisely in the final page of the proof of \cite[Lemma~5.5]{Hwang_Yun} this formula is written as 
$\delta \omega=E_2(u)\theta^1-E_1(u)\theta^2$. As already observed, $\omega$ corresponds to our $P$ in the formalism of the differential forms, and the codifferential $\delta$ is clearly related to the divergence through the formula $\delta \omega=-{\rm div} P$. 
\end{remark}

%

\section{Counterexamples to estimate~\eqref{eq:main}}
\label{sec:counterexample}

We work in dimension $3$ for simplicity, but similar counterexamples might be constructed in higher dimension as well. Consider local coordinates $\{r,x^1,x^2\}$ defined on an open set, a positive smooth function $\phi=\phi(r)$ and the warped product metric
$$
g\,=\,dr\otimes dr+\phi^2 (dx^1\otimes dx^1+ dx^2\otimes dx^2)\,.
$$
Let then $f \in \mathscr{C}^\infty(M)$ be a smooth function of the form $f = \psi \circ x^1$, for some smooth nonconstant real function $\psi$. Let us consider then a skew-symmetric 2-tensor field $P$ as in~\eqref{eq:P}, for some choice of $\lambda$.
In local coordinates, we have that the components of $P$ are given by
$$
P_{\a\b}\,=\,\lambda\left[\na_\a f\nana_{\b\eta}f-\na_\b f\nana_{\a\eta}f\right]g^{\eta\sigma}\na_\sigma f\,=\,
\frac{\lambda\,\psi'}{\phi^2}\left[\na_\a f\nana_{1\b}f-\na_\b f\nana_{1\a}f\right]\,,
$$
where the greek indexes are running in $\{r,1,2\}$.
Here and in what follows we will denote with $'$ the derivatives with respect to $x^1$ and with a dot the derivatives with respect to $r$.
The Christoffel symbols of the metric $g$ are as follows
$$
\Gamma_{rr}^r\,=\,\Gamma_{ri}^r\,=\Gamma_{rr}^i\,=\,\Gamma_{ij}^k\,=\,0\,,\ \ 
\Gamma_{ij}^r\,=\,-\phi\dot\phi\delta_{ij}\,,\ \ 
\Gamma_{ri}^j\,=\,\frac{\dot\phi}{\phi}\delta_i^j\,,
$$
where the latin indexes are running in $\{1,2\}$.
It then follows easily that the only nonzero components of the Hessian are
$$
\nana_{11}f\,=\,\psi''\,,\quad \nana_{1r}f\,=\,-\frac{\dot\phi}{\phi}\,\psi'\,,
$$ 
and that
$$
P\,=\,\lambda\frac{\dot\phi}{\phi^3}{(\psi')^3}\left(dr\otimes dx^1-dx^1\otimes dr\right)\,.
$$
Notice that we are in a setting similar to the one of Section~\ref{sec:issue}, except that our frame 
$$
\{\pa/\pa r,\pa/\pa x^1,\pa/\pa x^2\}
$$ 
is not orthonormal. Hence, to check that our $P$ has the structure prescribed in~\eqref{eq:frame}, one should write its local expression, with respect to an orthonormal frame.  This latter can be obtained setting $E_1=(1/\phi)\pa/\pa x^1$, $E_2=\pa /\pa r$, $E_3=(1/\phi)\pa/\pa x^2$. Its dual orthonormal co-frame is then given by $\theta^1=\phi d x^1, \theta^2=dr, \theta^3=\phi dx^2$. It is easy to check that this frame satisfies the properties described in Lemma~\ref{le:frame} and that 
$$
P\,=\,-\lambda\frac{\dot\phi}{\phi^4}{(\psi')^3}
\left(\theta^1\otimes\theta^2-\theta^2\otimes\theta^1\right)
\,.
$$
However, we prefer to perform our computations with respect to the frame fields induced by the local coordinates $(r,x^1,x^2)$. In this framework, it is easy to show that the only nonzero components of $\na P$ are 
\begin{align*}
\na_r P_{1r}\,&=\,-\left(\frac{\ddot\phi}{\phi^3}-4\frac{\dot\phi^2}{\phi^4}\right)\lambda \,(\psi')^3\,,
\\
\na_1 P_{1r}\,&=\,-\frac{\dot\phi}{\phi^3}(\lambda \, (\psi')^3)'\,,
\\
\na_2 P_{12}\,&=\,-\frac{\dot\phi^2}{\phi^2}\lambda\, (\psi')^3\,.
\end{align*}
It easily follows that 
$$
{\rm div}P\,=\,-\frac{\dot\phi}{\phi^5}(\lambda \, (\psi')^3)' \, dr \, + \, \lambda \, (\psi')^3\left(\frac{\ddot\phi}{\phi^3}-3\frac{\dot\phi^2}{\phi^4}\right)dx^1\,.
$$
Here it is possible to notice the discrepancy between our computations and formula~\eqref{eq:false}, as computing the right hand side of that formula would give
$$
-\frac{\dot\phi}{\phi^5}(\lambda \, (\psi')^3)' \,  dr \, +\, \lambda \, (\psi')^3\left(\frac{\ddot\phi}{\phi^3}-4\frac{\dot\phi^2}{\phi^4}\right)dx^1\,,
$$
which looks very similar, but does not correspond to the correct value of ${\rm div}P$.
Computing the squared norms of $\na P$  and ${\rm div}P$, one finally arrives at
$$
|\na P|^2-2|{\rm div}P|^2\,=\,4\frac{\lambda^2\dot\phi^2(\psi')^6}{\phi^8}\left(4\frac{\dot\phi^2}{\phi^2}-\frac{\ddot \phi}{\phi}\right) \, .
$$
To make this difference negative, it is then sufficient to specify a choice of the functions $\lambda, \psi$ and $\phi$ such that the right hand side is negative. In particular, it is sufficient to choose $\phi$ in such a way that the quantity in round brackets is negative. This can be achieved, for example, setting
$$
\phi\,=\,(r+c)^{-1/k},\qquad \hbox{for some $k>3$ and some $c>0$}\, . 
$$
It follows that, with this choice of $\phi$, for any $\lambda$ and any $f= \psi \circ x^1$, the estimate~\eqref{eq:mainP} does not hold. Hence, the lower bound~\eqref{eq:main} is false as well.


\section*{Appendix}

For completeness, let us point out the correct relation always holding between $|\na P|$ and $|{\rm div}P|$. Let $(M,g)$ be a $n$-dimensional Riemannian manifold, $n\geq 3$. 
As in Section~\ref{sec:issue}, we take a point $x$ with $|P|(x)\neq 0$ and we consider the local orthonormal frame $\{E_1,\dots,E_n\}$ provided by Lemma~\ref{le:frame}. We recall that, with respect to this frame, the tensor $P$ takes the following form
\begin{equation}
\label{eq:frame2}
P\,=\,u\left(\theta^1\otimes\theta^2-\theta^2\otimes\theta^1\right)\,.
\end{equation}
Exploiting the compatibility of $\na$ with the metric $g$, for any $i,j,k$ we have
$$
0=E_i \left(g(E_j,E_k)\right)=g(\na_{E_i}E_j,E_k)+g(E_j,\na_{E_i}E_k)\,,
$$
and in particular
$$
g(\na_{E_i}E_k,E_k)=0\,,\qquad g(\na_{E_i}E_i,E_k)\,=\,-g(E_i,\na_{E_i}E_k)\,=\,-g(E_i,[E_i,E_k])\,.
$$
We are now ready to compute the components of $\na P$. Since $P(E_i,E_j)=0$ whenever $\{i,j\}\neq\{1,2\}$, we have
\begin{align*}
\na_{E_i}P(E_1,E_2)\,&=\,E_i\left(P(E_1,E_2)\right)\,-\,P(\na_{E_i}E_1,E_2)\,-\,P(E_1,\na_{E_i}E_2)
\\
&=\,E_i(u)\,-\,g(\na_{E_i}E_1,E_1)P(E_1,E_2)\,-\,g(\na_{E_i}E_2,E_2)P(E_1,E_2)
\\
&=\,E_i(u)\,.
\end{align*}
Similarly, for any $k\geq 3$, we have
\begin{align*}
\na_{E_i}P(E_1,E_k)\,&=\,E_i(P(E_1,E_k))-\,P(\na_{E_i}E_1,E_k)\,-\,P(E_1,\na_{E_i}E_k)
\\
&=\,-\,g(\na_{E_i}E_k,E_2)P(E_1,E_2)
\\
&=\,-\,g(\na_{E_i}E_k,E_2)\,u\,\,,
\end{align*}
and
\begin{align*}
\na_{E_i}P(E_2,E_k)\,&=\,E_i(P(E_2,E_k))-\,P(\na_{E_i}E_2,E_k)\,-\,P(E_2,\na_{E_i}E_k)
\\
&=\,-\,g(\na_{E_i}E_k,E_1)P(E_2,E_1)
\\
&=\,g(\na_{E_i}E_k,E_1)\,u\,\,.
\end{align*}
Similarly, one computes $\na_{E_i}P(E_1,E_1)=\na_{E_i}P(E_2,E_2)=0$ and $\na_{E_i}P(E_j,E_k)=0$ whenever $j,k$ are $\geq 3$.
It is now easy to compute the divergence of $P$:
\begin{equation*}
\begin{aligned}
({\rm div}\,P)(E_1)\,&=\,-E_2(u)\,+\,\sum_{i=3}^n\langle \na_{E_i}E_i\,|\,E_2\rangle \,u
\,=\,-E_2(u)\,+\,\sum_{i=3}^n\langle E_i\,|\,[E_2,E_i]\rangle \,u\,,
\\
({\rm div}\,P)(E_2)\,&=\,E_1(u)\,-\,\sum_{i=3}^n\langle \na_{E_i}E_i\,|\,E_1\rangle \,u\,
\,=\,E_1(u)\,-\,\sum_{i=3}^n\langle E_i\,|\,[E_1,E_i]\rangle \,u\,,
\\
({\rm div}\,P)(E_i)
\,&=\,-g(\na_{E_1}E_i,E_2)\,u+g(\na_{E_2}E_i,E_1)\,u
\,,\quad i\geq 3\,.
\end{aligned}
\end{equation*}
Using the inequality $(\sum_{i=1}^k x_i)^2\leq k\sum_{i=1}^k x_i^2$, a simple calculation then gives
\begin{align*}
\frac{|{\rm div} P|^2}{n-1}\,&\leq\,\sum_{k=1}^2\left[E_k(u)^2+\sum_{i= 3}^n\langle E_i\,|\,[E_i,E_k]\rangle^2 u^2\right]+\frac{2}{n-1}\sum_{i= 3}^n\left[\langle \na_{E_1}E_i\,|\,E_2\rangle^2+\langle \na_{E_2}E_i\,|\,E_1\rangle^2 \right]u^2
\\
&\leq\,\sum_{k=1}^2\left[E_k(u)^2+\sum_{i= 3}^n\langle E_i\,|\,[E_i,E_k]\rangle^2 u^2\right]+\sum_{i= 3}^n\left[\langle \na_{E_1}E_i\,|\,E_2\rangle^2+\langle \na_{E_2}E_i\,|\,E_1\rangle^2 \right]u^2\,.
\end{align*}
On the other hand
\begin{align*}
\frac{1}{2}|\na P|^2&\geq\,\sum_{k=1}^2\left[\left(\na_{E_k}P(E_1,E_2)\right)^2+\sum_{i= 3}^n\left(\na_{E_i}P(E_k,E_i)\right)^2+\sum_{i= 3}^n\left(\na_{E_k}P(E_k,E_i)\right)^2\right]
\\
&=\,\sum_{k=1}^2\left[E_k(u)^2+\sum_{i= 3}^n\langle E_i\,|\,[E_i,E_k]\rangle^2 u^2\right]+\sum_{i= 3}^n\left[\langle \na_{E_1}E_i\,|\,E_2\rangle^2+\langle \na_{E_2}E_i\,|\,E_1\rangle^2 \right]u^2\,.
\end{align*}

In conclusion, we have shown the following.
\begin{proposition}
Let $(M,g)$ be a $n$-dimensional Riemannian manifold, $n\geq 3$. Let $f\in\mathscr{C}^\infty(M)$ and let $P$ be the tensor defined by~\eqref{eq:P}. Then, at any point of $M$ it holds
\begin{equation}
\label{eq:truebound}
|\na P|^2\,\geq\,\frac{2}{n-1}\,|{\rm div}P|^2\,.
\end{equation}
\end{proposition}

\begin{proof}
Estimate~\eqref{eq:truebound} follows immediately from the computations above at any point where $P$ has the form~\eqref{eq:P}, that is, at any point where $|P|\neq 0$. Let then $x$ be a point where $|P|=0$. If $|P|$ vanishes identically in a neighborhood of $x$, then $|\na P|=|{\rm div}\,P|=0$ in that neighborhood, and inequality~\eqref{eq:truebound} is trivially satisfied. Otherwise there exists a sequence of points $x_i$ converging to $x$ with $|P|(x_i)\neq 0$. Since estimate~\eqref{eq:truebound} holds at the points $x_i$, then it must hold at $x$ as well by continuity.
\end{proof}

%

\subsection*{Acknowledgements}
{\em The authors would like to thank R. Beig, P. T. Chru\'sciel and W. Simon for stimulating discussions about the classification of static vacuum spacetimes. The authors are members of the Gruppo Nazionale per l'Analisi Matematica, la Probabilit\`a e le loro Applicazioni (GNAMPA) of the Istituto Nazionale di Alta Matematica (INdAM).
}


\bibliographystyle{plain}
\bibliography{biblio}

\end{document}